\documentclass[11pt]{article}

\usepackage{amsfonts,amsmath,latexsym,color,epsfig}
\newtheorem{theorem}{Theorem}[section]
\newtheorem{lemma}{Lemma}[section]

\newtheorem{corollary}{Corollary}[section]

\newenvironment{proof}
      {\medskip\noindent{\bf Proof:}\hspace{1mm}}
      {\hfill$\Box$\medskip}

\input{epsf}
\makeatletter
\def\Ddots{\mathinner{\mkern1mu\raise\p@
\vbox{\kern7\p@\hbox{.}}\mkern2mu
\raise4\p@\hbox{.}\mkern2mu\raise7\p@\hbox{.}\mkern1mu}}
\makeatother

\def\e{\epsilon}

\title{\vspace{-0.7cm}An extremal theorem in the hypercube}
\author{David Conlon\thanks{St John's College, Cambridge CB2 1TP, United Kingdom.
E-mail: {\tt
D.Conlon@dpmms.cam.ac.uk}. Supported by a Junior Research Fellowship
at St John's College.}}
\date{}
\begin{document}
\maketitle

\begin{abstract}
The hypercube $Q_n$ is the graph whose vertex set is $\{0,1\}^n$ and where two vertices are adjacent if they differ in exactly one coordinate. For any subgraph $H$ of the cube, let $\textrm{ex}(Q_n, H)$ be the maximum number of edges in a subgraph of $Q_n$ which does not contain a copy of $H$. We find a wide class of subgraphs $H$, including all previously known examples, for which $\textrm{ex}(Q_n, H) = o(e(Q_n))$. In particular, our method gives a unified approach to proving that $\textrm{ex}(Q_n, C_{2t}) = o(e(Q_n))$ for all $t \geq 4$ other than $5$. 
\end{abstract}

\section{Introduction}

Given two graphs $G$ and $H$, $\textrm{ex}(G, H)$ is the maximum number of edges in a subgraph of $G$ which does not contain a copy of $H$. The study of such numbers was initiated by Paul Tur\'an \cite{T41} when he determined the maximum number of edges in a graph on $n$ vertices which does not contain a clique of size $r$, that is, $\textrm{ex}(K_n, K_r)$. For any graph $H$, the value of $\textrm{ex}(K_n, H)$ is known \cite{ES66} to be intimately connected to the chromatic number of $H$. In particular, it is known \cite{KST54} that $\textrm{ex}(K_n, H) = o(e(K_n))$ if and only if $H$ is bipartite. In this note, we will be similarly interested in determining subgraphs $H$ of the hypercube $Q_n$ for which $\textrm{ex}(Q_n, H) = o(e(Q_n))$.

The hypercube $Q_n$ is the graph whose vertex set is $\{0,1\}^n$ and where two vertices are adjacent if they differ in exactly one coordinate. This graph has $2^n$ vertices and, being $n$-regular, $2^{n-1} n$ edges. Erd\H{o}s \cite{E84} was the first to draw attention to Tur\'an-type problems in the cube when he asked how many edges a $C_4$-free subgraph of the cube can contain. He conjectured that the answer is $\left(\frac{1}{2} + o(1)\right) e(Q_n)$ and offered \$$100$ for a solution. 

Improving a long-standing result of Chung \cite{Ch92}, Thomason and Wagner \cite{TW09} recently gave an upper bound of roughly $ 0.6226 e(Q_n)$ for $\textrm{ex}(Q_n, C_4)$. This remains a long way from the lower bound of $\frac{1}{2} (n + \sqrt{n}) 2^{n-1}$ (valid when $n$ is a power of $4$), due to Brass, Harborth and Nienborg \cite{BHN95}. A related result of Bialostocki \cite{B83} implies that if the edges of the cube are $2$-coloured and neither colour class contains a $C_4$, then the number of edges in each colour is at most $\frac{1}{2} (n + \sqrt{n}) 2^{n-1}$. Therefore, at least in some sense, the construction of Brass, Harborth and Nienborg is optimal.

Erd\H{o}s \cite{E84, E90} also posed the problem of determining $\textrm{ex}(Q_n, C_{2t})$ for all $t > 2$, suggesting that perhaps $o(e(Q_n))$ was sufficient. This is false for $t = 3$. Indeed, Chung \cite{Ch92} and Brouwer, Dejter and Thomassen \cite{BDT93} found $4$-colourings of the cube without any monochromatic $C_6$ and, later, Conder \cite{Co93} found a $3$-colouring with the same property. This implies that $\textrm{ex}(Q_n, C_6) \geq \frac{1}{3} e(Q_n)$. On the other hand, it is known \cite{Ch92} that $\textrm{ex}(Q_n, C_6) \leq \left(\sqrt{2} - 1 +o(1)\right) e(Q_n)$.

For even $t \geq 4$, Chung \cite{Ch92} justified Erd\H{o}s' intuition by showing that $\textrm{ex}(Q_n, C_{2t}) = o(e(Q_n))$. However, until very recently, it was unknown whether a similar result holds for odd values of $t$. This problem was almost completely resolved by F\"uredi and \"Ozkahya \cite{FO09, FO209}, who showed that, for odd $t \geq 7$, 
$\textrm{ex}(Q_n, C_{2t}) = o(e(Q_n))$. The only case that now remains unresolved is $C_{10}$.  

Some progress on this problem has been made by Alon, Radoi\v{c}i\'c, Sudakov and Vondr\'ak \cite{ARSV06}, who proved that in any $k$-colouring of the edges of $Q_n$, for $n$ sufficiently large, there are monochromatic copies of $C_{10}$. Therefore, unlike $C_4$ and $C_6$, a counterexample to the conjecture that $\textrm{ex}(Q_n, C_{10}) = o(e(Q_n))$ cannot come from a colouring. On the other hand, Alon et al. \cite{ARSV06} and Axenovich and Martin \cite{AxM06} both note that there is a $4$-colouring of $Q_n$ which does not contain an induced monochromatic copy of $C_{10}$. 

More generally, Alon et al. \cite{ARSV06} gave a characterisation of all subgraphs $H$ of the cube which are Ramsey, that is, such that every $k$-edge-colouring of a sufficiently large $Q_n$ contains a monochromatic copy of $H$. It would be nice to have a similar characterisation of all graphs $H$ for which $\textrm{ex}(Q_n, H) = o(e(Q_n))$. Unfortunately, even deciding whether this is true for $C_{10}$ seems very difficult, so a necessary and sufficient condition is probably well beyond our grasp. Nevertheless, the purpose of this note is to present a natural sufficient condition. To state the condition precisely, we will need some notation.

Recall that the vertex set of the hypercube $Q_n$ is $\{0,1\}^n$. Any vertex, such as $[01101]$ in $Q_5$, can be written as a sequence of bits. Any edge can be written uniquely by a sequence such as $[01{*}11]$. The missing bit, known hereafter as the flip-bit, tells us that the edge connects the vertex where that missing bit is equal to $0$ with the vertex where it is $1$. In this case, $[01011]$ is connected to $[01111]$.  

We say that a subgraph $H$ of the hypercube has a $k$-partite representation if there exists $l$ such that 
\begin{itemize}
\item
$H$ is a subgraph of $Q_l$;

\item
every edge $e = [a_1 a_2 \dots a_l]$ in $H$ has exactly $k$ non-zero bits ($k-1$ ones and a flip-bit); 

\item
there exists a function $\sigma: [l] \rightarrow [k]$ such that, for each edge $e$, the image $\{\sigma(i_1), \dots, \sigma(i_k)\}$ of the set of non-zero bits $\{a_{i_1}, \dots, a_{i_k}\}$ of $e$ under $\sigma$ is $[k]$, that is, no two non-zero bits have the same image. 
\end{itemize}

Let $\mathcal{H}$ be the $k$-uniform hypergraph on vertex set $[l]$ with edge set $\{\tau(e) : e \in E(H)\}$, where the function $\tau$ is defined by mapping the set of non-zero bits $\{a_{i_1}, \dots, a_{i_k}\}$ of $e$ to the subset $\{i_1, \dots, i_k\}$ of $[l]$. We refer to $\mathcal{H}$ as a representation of $H$. By the definition of a $k$-partite representation, $\mathcal{H}$ must, unsurprisingly, be $k$-partite. The subsets $\sigma^{-1}(1), \dots, \sigma^{-1}(k)$ of $[l]$ will be referred to as the partite sets of the representation.

As an initial example, note that $C_8$ has a $2$-partite representation. This may be seen by taking $\sigma(1) = \sigma(3) = 1$ and $\sigma(2) = \sigma(4) = 2$ in the following example. 
\begin{quote}
\begin{tabular}{llllllllll}
$e_1 =$ & [1 & $*$ & 0 & 0] \hspace{30mm} & $e_5 =$ & [0 & 0 & 1 & $*$]\\
$e_2 =$ & [$*$ & 1 & 0 & 0] & $e_6 =$ & [0 & 0 & $*$ & 1]\\
$e_3 =$ & [0 & 1 & $*$ & 0] & $e_7 =$ & [$*$ & 0 & 0 & 1]\\
$e_4 =$ & [0 & $*$ & 1 & 0] & $e_8 =$ & [1 & 0 & 0 & $*$]
\end{tabular}
\end{quote}
Similarly, for any even $t \geq 4$, it is easy to see that $C_{2t}$ has a $2$-partite representation, namely, the cycle $C_t$ of length $t$. This observation is at the core of Chung's proof that $\textrm{ex}(Q_n, C_{2t}) = o(e(Q_n))$. 

For odd values of $t$, this is not true. However, for $t \geq 7$, these graphs do admit a $3$-partite representation. For example, in the case of $C_{14}$, the following representation works with $\sigma(1) = \sigma(4) = 1$, $\sigma(2) = \sigma(6) = 2$ and $\sigma(3) = \sigma(5) = \sigma(7) = 3$.
\begin{quote}
\begin{tabular}{llllllllllllllll}
$e_1 =$ & [1 & 1 & $*$ & 0 & 0 & 0 & 0] \hspace{12mm} & $e_8 =$ & [1 & $*$ & 0 & 0 & 1 & 0 & 0]\\
$e_2 =$ & [$*$ & 1 & 1 & 0 & 0 & 0 & 0] & $e_9 =$ & [1 & 0 & 0 & 0 & 1 & $*$ & 0]\\
$e_3 =$ & [0 & 1 & 1 & $*$ & 0 & 0 & 0] & $e_{10} =$ & [1 & 0 & 0 & 0 & $*$ & 1 & 0]\\
$e_4 =$ & [0 & 1 & $*$ & 1 & 0 & 0 & 0] & $e_{11} =$ & [1 & 0 & 0 & 0 & 0 & 1 & $*$]\\
$e_5 =$ & [0 & 1 & 0 & 1 & $*$ & 0 & 0] & $e_{12} =$ & [1 & 0 & 0 & 0 & 0 & $*$ & 1]\\
$e_6 =$ & [0 & 1 & 0 & $*$ & 1 & 0 & 0] & $e_{13} =$ & [1 & $*$ & 0 & 0 & 0 & 0 & 1]\\
$e_7 =$ & [$*$ & 1 & 0 & 0 & 1 & 0 & 0] & $e_{14} =$ & [1 & 1 & 0 & 0 & 0 & 0 & $*$]
\end{tabular}
\end{quote}
More generally, for all odd $t \geq 7$, $C_{2t}$ has a $3$-partite representation which is close to a hypergraph cycle. This representation is particularly simple when $t$ is a multiple of $3$, when it corresponds to the tight cycle, that is, the $3$-uniform hypergraph on $t$ vertices $v_1, \dots, v_t$ whose edge set is 
\[\{v_1 v_2 v_3, \ v_2 v_3 v_4, \ v_3 v_4 v_5, \ \dots, \ v_{t-1} v_t v_1, \ v_t v_1 v_2\}.\]

On the other hand, $C_4$, $C_6$ and $C_{10}$ do not admit $k$-partite representations for any $k$. The importance of these considerations lies in the following theorem.

\begin{theorem} \label{IntroMain}
Let $H$ be a fixed subgraph of the cube. If, for some $k$, $H$ admits a $k$-partite representation, then
\[\textrm{ex}(Q_n, H) = o(e(Q_n)).\]
\end{theorem}

We therefore have a unified proof that $\textrm{ex}(Q_n, C_{2t}) = o(e(Q_n))$ for all $t \geq 4$ other than $5$. However, the theorem plainly applies to a much wider class of graphs than cycles. If, for example, we add the edge $e_{15} = [1100{*}00]$ to the $C_{14}$ given above, we get a $3$-partite representation of a $C_{14}$ with a long diagonal, reproving another result due to F\"uredi and \"Ozkahya \cite{FO09}. 

For the sake of clarity of presentation, we will systematically omit floor and ceiling signs whenever they are not crucial. We also do not make any serious attempt to optimize absolute constants in our statements and proofs.

\section{Subgraphs of the cube with zero Tur\'an density}

We will need a simple estimate stating that almost all vertices in the cube $Q_n$ have roughly the same number of zeroes and ones. Since there is a one-to-one correspondence between vertices of $Q_n$ and subsets of the set $[n] = \{1, 2, \dots, n\}$, the required estimate follows from the following lemma.

\begin{lemma} \label{DenseDiag}
The number of subsets of $[n]$ containing fewer than $n/4$ or more than $3n/4$ elements is at most $(1.9)^n n$.  
\end{lemma}

\begin{proof}
By symmetry, it is easy to see that the number of subsets of $[n]$ with fewer than $n/4$ elements is the same as the number of subsets with more than $3n/4$ elements. Therefore, the number of subsets with fewer than $n/4$ elements or more than $3n/4$ elements is at most
\[2 \sum_{i=0}^{\lfloor n/4 \rfloor} \binom{n}{i} \leq n \binom{n}{n/4} \leq (4 e)^{n/4} n,\]
where we used the estimate $j! \geq \left(j/e\right)^j$ with $j = n/4$. The result follows since $(4 e)^{1/4} \leq 1.9$.
\end{proof}

\begin{theorem} \label{Main}
Let $H$ be a subgraph of the cube with $k$-partite representation $\mathcal{H}$. Suppose $\textrm{ex}(K_n^{(k)}, \mathcal{H}) \leq \alpha n^k$. Then
\[\textrm{ex}(Q_n, H) = O(\alpha^{1/k} 2^n n).\]
\end{theorem}

\begin{proof} Suppose that we have a subgraph $G$ of $Q_n$ with density $\e = 16 k \alpha^{1/k}$. We will prove that, for $n$ sufficiently large, $G$ must contain a copy of $H$. By Lemma \ref{DenseDiag}, the number of vertices containing fewer than $n/4$ or more than $3 n/4$ ones is at most $(1.9)^n n$. Therefore, since each vertex has maximum degree $n$, the number of edges between levels $i$ and $i+1$, added over all $i$ for which $i < n/4$ or $i \geq 3n/4$, is at most $(1.9)^n n^2$. Since there are $2^{n-1} n$ edges in all, the density contribution of these edges is at most $2 (0.95)^n n$, which, for $n$ sufficiently large, is less than $\frac{\e}{2}$. Therefore, $G$ has a density of at least $\frac{\e}{2}$ concentrated between levels $n/4$ and $3n/4$. In particular, there exists some $j$ with $n/4 \leq j < 3n/4$ such that the density of edges in $G$ between levels $j$ and $j+1$ is at least $\frac{\e}{2}$. 

Every edge between levels $j$ and $j+1$ may be represented by a collection of $j$ ones and a flip-bit. By the choice of $j$, there are at least  $\frac{\e}{2} (n - j) \binom{n}{j} \geq \frac{\e j}{2} \binom{n}{j+1}$ edges of $G$ between levels $j$ and $j+1$. Given a subset $J$ of $[n]$ of size $j + 1$, let $d(J)$ be the number of edges for which the union of the flip-bit and the set of ones is $J$. Since $\sum d(J)$ is the number of edges between levels $j$ and $j+1$, $\mathbb{E} (d(J)) \geq \frac{\e j}{2}$. Therefore, by convexity of the function $\binom{x}{k}$, we have
\[\sum_{J \subset [n], |J| = j+1} \binom{d(J)}{k} \geq \binom{n}{j+1} \binom{\mathbb{E} (d(J))}{k} \geq \binom{n}{j+1} \binom{\e j/2}{k}.\]
For any subset $J$ of $[n]$ of size $j+1$, let $D(J) \subset J$ be the set of positions where replacing an element of $J$ with a flip-bit yields an edge. The previous equation tells us that there are at least $\binom{n}{j+1} \binom{\e j/2}{k}$ pairs $(I,J)$ for which $J$ has size $j+1$, $I \subset D(J)$ and $|I| = k$. Since there are $\binom{n}{j + 1 - k}$ ways of choosing a subset of $n$ of size $j + 1 - k$, we see that there must be some set $S$ of size $j+1-k$ for which at least 
\[\binom{n}{j+1} \binom{\e j/2}{k}/\binom{n}{j+1-k} = \binom{n-j-1+k}{k} \binom{\e j/2}{k} /\binom{j+1}{k}\]
pairs $(I, J)$ have $J \char92 I = S$. Here we used the identity $\binom{n}{j + 1} \binom{j + 1}{k} = \binom{n}{j + 1 - k} \binom{n - j - 1 + k}{k}$. 

Fixing $S$, we see that the pair $(I, J)$ is uniquely determined by the choice of $I$. Let $\mathcal{I}$ be the $k$-uniform hypergraph whose edges are the sets $I$ taken from these pairs. Since $n \geq \alpha^{-1/k} = 16 k/\e$ and $j \geq n/4$, we have $j \geq 4 k/\e$. This in turn implies that $\binom{\e j/2}{k}/\binom{j+1}{k} \geq \left(\frac{\e}{4}\right)^k$. Therefore, since $j < 3n/4$, the number of edges in $\mathcal{I}$ is at least
\[\left(\frac{\e}{4}\right)^k \binom{n - j -1 + k}{k} \geq \left(\frac{\e}{16}\right)^k \frac{n^k}{k!}.\]
Hence, since $\left(\frac{\e}{16}\right)^k \geq k!\alpha$, the hypergraph $\mathcal{I}$ contains a copy of $\mathcal{H}$. Suppose that $\mathcal{H}$ is defined on vertex set $[l]$ and the mapping $g: [l] \rightarrow [n]$ describes the embedding of $\mathcal{H}$ in $\mathcal{I}$. 

We define a map $f: Q_l \rightarrow Q_n$ by mapping $[a_1 \dots a_l]$, with non-zero bits $a_{i_1}, \dots, a_{i_r}$, to $[b_1 \dots b_n]$, where $b_i = 1$ if and only if $i \in S \cup \{g(i_1), \dots, g(i_r)\}$. It is straightforward to verify that this is a graph isomorphism between $Q_l$ and $f(Q_l)$. Moreover, for every edge $e = uv \in H$, the edge $f(u)f(v)$ is in $G$. 
To see this, suppose that the non-zero bits of $e$ are $a_{i_1}, \dots, a_{i_k}$ and $a_{i_{\ell}}$ is the flip-bit. Let $J = S \cup \{g(i_1), \dots, g(i_k)\}$. By construction, $I = \{g(i_1), \dots, g(i_k)\} \in \mathcal{I}$ and, therefore, $I \subset D(J)$. 
Hence, by the definition of $D(J)$, the edge formed by replacing $b_{g(i_{\ell})}$ in $f(u)$ with a flip-bit is in $G$. But this edge is just the edge between $f(u)$ and $f(v)$. We therefore have an embedding of $H$ in $G$, completing the proof.
\end{proof}

To complete the proof of Theorem \ref{IntroMain}, we only need to apply the following classical result of Erd\H{o}s \cite{E64} regarding the extremal number of complete $k$-partite $k$-uniform hypergraphs.

\begin{lemma} \label{Erdos}
Let $K_k^{(k)}(s_1, \dots, s_k)$ be the complete $k$-partite $k$-uniform hypergraph with partite sets of size $s_1, \dots, s_k$. Then 
\[\textrm{ex}(K_n^{(k)}, K_k^{(k)} (s_1, \dots, s_k)) = O(n^{k - \delta}),\]
where $\delta = \left(\prod_{i=1}^{k-1} s_i \right)^{-1}$.
\end{lemma}

This yields the following, more precise, version of Theorem \ref{IntroMain}.

\begin{corollary}
Let $H$ be a subgraph of the cube which admits a $k$-partite representation, where the partite sets have sizes $s_1, \dots, s_k$. Then 
\[\textrm{ex}(Q_n, H) = O(2^n n^{1 - \frac{\delta}{k}}),\]
where $\delta = \left(\prod_{i=1}^{k-1} s_i \right)^{-1}$.
\end{corollary}

\section{Concluding remarks}

$\bullet$ It still remains to decide whether $\textrm{ex}(Q_n, C_{10}) = o(e(Q_n))$. It even remains open to decide whether there is any graph $H$ which is Ramsey with respect to the cube but which does not have zero Tur\'an density. We have also been unable to determine whether there are graphs with zero Tur\'an density which do not have $k$-partite representations. We conjecture that these three collections of graphs, those with $k$-partite representations, those with zero Tur\'an density and those which are Ramsey with respect to the cube, are all distinct.

$\bullet$ A quantitative version of Chung's result regarding the appearance of cycles in cubes states that, for $t$ even, $\textrm{ex}(Q_n, C_{2t}) = O(2^n n^{\frac{1}{2} + \frac{1}{t}})$. This result follows easily from Theorem \ref{Main} and the fact that the $2$-partite representation of $C_{2t}$ is the graph $C_t$. The one additional ingredient necessary to complete the proof is the Bondy-Simonovits theorem \cite{BS74}, that $\textrm{ex}(K_n, C_t) = O(n^{1 + \frac{2}{t}})$ for $t$ even.

On the other hand, for every $t$, an application of the Lov\'asz local lemma implies that there is a subgraph $G_t$ of $Q_n$ with $\Omega(2^n n^{\frac{1}{2} + \frac{1}{2t}})$ edges which does not contain a copy of $C_{2t}$. For $t$ even, we believe that the upper bound is tight but have been unable to make any progress towards proving this.

$\bullet$ For odd values of $t$, the behaviour of the function $\textrm{ex}(Q_n, C_{2t})$ is even more obscure. F\"uredi and \"Ozkahya \cite{FO09, FO209} give an upper bound of the form $\textrm{ex}(Q_n, C_{2t}) = O(2^n n^{\frac{6}{7}})$ for all odd values of $t$ with $t \geq 7$. If the $3$-partite $3$-uniform graph $\mathcal{E}_t$ that represents $C_{2t}$ satisfies $\textrm{ex}(K_n^{(3)}, \mathcal{E}_t) = O(n^{2+\delta})$, then Theorem \ref{Main} would imply that $\textrm{ex}(Q_n, C_{2t}) = O(2^n n^{\frac{2}{3} + \frac{\delta}{3}})$, improving the result of F\"uredi and \"Ozkahya for $\delta$ sufficiently small. While such an improved estimate almost certainly holds for large $t$, we have not pursued this direction.

\vspace{0.1cm} \noindent {\bf Acknowledgments.}\, I would like to thank Eoin Long for reading carefully through an earlier
version of this note.


\begin{thebibliography}{}

\bibitem{AKS07}
{N. Alon, A. Krech and T. Szab\'o,} {Tur\'an's theorem in the hypercube,} {\it SIAM J. Discrete Math.} {\bf 21} (2007), 66--72.

\bibitem{ARSV06}
{N. Alon, R. Radoi\v{c}i\'c, B. Sudakov and J. Vondr\'ak,} {A Ramsey-type result for the hypercube,} {\it J. Graph Theory} {\bf 53} (2006), 196--208.

\bibitem{AxM06}
{M. Axenovich and R. Martin,} {A note on short cycles in the hypercube,} {\it Discrete Math.} {\bf 306} (2006), 2212--2218.

\bibitem{B83}
{A. Bialostocki,} {Some Ramsey type results regarding the graph of the $n$-cube,} {\it Ars Combin.} {\bf 16} (1983), 39--48.

\bibitem{BS74}
{J. A. Bondy and M. Simonovits,} {Cycles of even length in graphs,} {\it J. Combin. Theory Ser. B} {\bf 16} (1974), 97--105. 

\bibitem{BHN95}
{P. Brass, H. Harborth and H. Nienborg,} {On the maximum number of edges in a $C_4$-free subgraph of $Q_n$,} {\it J. Graph Theory} {\bf 19} (1995), 17--23.

\bibitem{BDT93}
{A. E. Brouwer, I. J. Dejter and C. Thomassen,} {Highly symmetric subgraphs of hypercubes,} {\it J. Algebraic Combin.} {\bf 2} (1993), 25--29.

\bibitem{Ch92}
{F. Chung,} {Subgraphs of a hypercube containing no small even cycles,} {\it J. Graph Theory} {\bf 16} (1992), 273--286.

\bibitem{Co93}
{M. Conder,} {Hexagon-free subgraphs of hypercubes,} {\it J. Graph Theory} {\bf 17} (1993), 477--479.

\bibitem{E64}
{P. Erd\H{o}s,} {On extremal problems of graphs and generalized graphs,} {\it Israel J. Math.} {\bf 2} (1964), 183--190.

\bibitem{E84}
{P. Erd\H{o}s,} {On some problems in graph theory, combinatorial analysis and combinatorial number theory,} {in: Graph Theory and Combinatorics (Cambridge, 1983),} {Academic Press, London,} 1984, 1--17.

\bibitem{E90}
{P. Erd\H{o}s,} {Some of my favourite unsolved problems,} in: {\bf A tribute to Paul Erd\H{o}s,} Cambridge University Press, 1990, 467--478. 

\bibitem{ES66}
{P. Erd\H{o}s and M. Simonovits,} {A limit theorem in graph theory,} {\it Studia Sci. Math. Hungar.} {\bf 1} (1966), 51--57.

\bibitem{FO09}
{Z. F\"uredi and L. \"Ozkahya,} {On 14-cycle-free subgraphs of the hypercube,} {\it Combin. Probab. Comput.} {\bf 18} (2009), 725--729.

\bibitem{FO209}
{Z. F\"uredi and L. \"Ozkahya,} {On even-cycle-free subgraphs of the hypercube,} {\it Electronic Notes in Discrete Mathematics} {\bf 34} (2009), 515--517. 

\bibitem{KST54}
{T. K\H{o}v\'ari, V. S\'os and P. Tur\'an,} {On a problem of K. Zarankiewicz,} {\it Colloq. Math.} {\bf 3} (1954), 50--57.

\bibitem{TW09}
{A. Thomason and P. Wagner,} {Bounding the size of square-free subgraphs of the hypercube,} {\it Discrete Math.} {\bf 309} (2009), 1730--1735.

\bibitem{T41}
{P. Tur\'an,} {On an extremal problem in graph theory (in Hungarian),} {\it Mat. Fiz. Lapok} {\bf 48} (1941), 436--452.

\end{thebibliography}
\end{document}